\documentclass[11pt]{amsart}
\usepackage{amsmath,amsfonts,amssymb,amsthm,a4} 
\usepackage{mathrsfs,enumerate}
\usepackage{mathpazo}
\usepackage{color}
\usepackage[dvipsnames]{xcolor}

\numberwithin{equation}{section}
\newcommand{\R}{\mathbb{R}}

\newcommand{\C}{\mathbb{C}}

\newcommand{\e}{\varepsilon}
\newcommand{\eps}{\varepsilon}

\newcommand{\ol}{\overline}

\newtheorem{thm}{Theorem}[section] 
\theoremstyle{definition}
\newtheorem{lem}[thm]{Lemma}

\author{Matthias Kurzke}
\address{School of Mathematical Sciences, University of Nottingham, Nottingham NG7~2RD, United Kingdom}
\email{matthias.kurzke@nottingham.ac.uk}
\date{19 October 2018}
\title[Correlation coefficient]
{The correlation coefficient for vortices in the plane}


\begin{document}
\begin{abstract}
For point vortices in the plane, we consider the correlation coefficient of Ovchinnikov and Sigal.
Generalising a result by Esposito, we show that it vanishes for all vortex equilibria.
 \end{abstract}
\maketitle 


\section{Introduction}
The interaction of $N$ point vortices in the plane is governed by the Kirchhoff-Onsager functional
\begin{equation}\label{eq:Wis}
W(a_1,d_1,\dots, a_N, d_N)=\sum_{j\neq k} d_j d_k\log\frac{1}{|a_j-a_k|} 
\end{equation}
where $a_j\in \R^2$ are the (pairwise distinct) vortex positions and $d_j\in\mathbb{R}\setminus\{0\}$ the corresponding circulations. 

The function $W$ appears in many contexts, for example it also describes the Coulombian interaction energy between $2D$ 
charges (or charged parallel $3D$ lines). Our interest in it comes from the fact that it 
governs the interaction of Ginzburg-Landau vortices in the plane (in that case 
we have $d_j\in\mathbb{Z}$ corresponding to the topological degree of the vortex at $a_j$), and is known as the
``renormalized energy'' of Bethuel, Brezis and H\'elein \cite{BBH}. Its gradient 
$\nabla W=(\frac{\partial W}{\partial a_1},\dots, \frac{\partial W}{a_N})\in\R^{2N}$ 
 has been shown to govern the dynamics of related
time-dependent problems.

One of the main open problems in the
theory of Ginzburg-Landau vortices, raised by Brezis, Merle and Rivi\`ere \cite{BMR}, is to classify the solutions of
\[
-\Delta u = u (1-|u|^2)
\]
where $u:\R^2\to \C$ is such that $|u(x)|\to 1$ as $|x|\to \infty$ and $\operatorname{deg}(\frac{u}{|u|};\partial B_R)=n\in \mathbb{Z}$ for 
$R$ large. It is known that solutions of the form $u(re^{i\theta})=f_n(r) e^{in\theta}$ exist, corresponding to a single vortex of degree $n$
at the origin,
 but it is open whether these solutions
are unique up to translations and rotations. Ovchinnikov and Sigal \cite{OvSi} devised a program to construct non-radial solutions, starting 
from a configuration $(a_j, d_j)$ with $\nabla W=0$ and such that the \emph{correlation coefficient} 
\begin{equation}\label{eq:Ais}
A(a_1,d_1,\dots, a_N, d_N):=\int_{\R^2 } \left(\left|\sum_{j=1}^N \frac{d_j}{z-a_j}\right|^4-\sum_{j=1}^N \frac{d_j^4}{|z-a_j|^4} 
 \right)dxdy
\end{equation}
is positive. Here $z=x+iy\in\C\simeq\R^2$. 
Using some asymptotics and scaling arguments, Ovchinnikov and Sigal deduced the existence of  a critical value $s>0$ such that 
solutions with vortices near
$(sa_1,\dots, sa_N)$ can be expected to exist. These solutions would break the radial symmetry and answer the question of 
Brezis, Merle and Rivi\`ere.

Closer inspection of the definition \eqref{eq:Ais}
shows the function being integrated is not in $L^1$ near the points $a_\ell$, so the integral must be 
understood as a principal value $A=\lim_{\eps\to 0}A_\eps$ with $A_\eps$ defined as 
\[
A_\eps(a_1,d_1,\dots, a_N, d_N)=\int_{\R^2\setminus \bigcup_{\ell=1}^N B_\eps(a_\ell)} \Biggl(\Bigl|\sum_{j=1}^N \frac{d_j}{z-a_j}\Bigr|^4-\sum_{j=1}^N \frac{d_j^4}{|z-a_j|^4} 
 \Biggr)dxdy.
\]
For the vortex equilibria considered by Ovchinnikov and Sigal,
Esposito \cite{Espo} later  showed that the correlation coefficient vanishes.
These equilibria 
have explicitly known vortex positions and are invariant under a finite group of symmetries. 
The calculations in \cite{OvSi} and \cite{Espo} make explicit use of this symmetry. 
However, many other vortex equilibria exist, and a large number of them
have no obvious symmetry whatsoever. For examples,  solutions where $d_i=\pm 1$ can be found by a polynomial method
as zeroes of successive Adler-Moser polynomials, compare \cite{ArefVortexCrystals,  ClarksonVortices, AdlerMoser}. 
Related polynomial methods have also been used recently in the construction of solutions of the Gross-Pitaevskii equation  \cite{LiuWei}.

Given the use of symmetry in the calculation of the correlation coefficient by Esposito,
 it is  natural to consider whether the vanishing of the correlation coefficient is related to the symmetry of the equilibria.
In the present note, we consider the correlation coefficient for general 
real-valued nonzero weights $d_j$. Our result (Theorem~\ref{thm:t}) is that $A=0$ for every configuration of vortices with $\nabla W=0$,
thus generalising Esposito's result to \emph{all} possible equilibrium configurations. Instead of relying on symmetry, our proof 
uses a decomposition of the integrands that significantly simplifies the computation.

\section{Computing the correlation coefficient}

We start with an algebraic observation about equilibria. It appears for example in  \cite{ONeilSta} or 
\cite{BethuelOrlandiSmetsCMP}. The argument using residues can be found in \cite{CoxON}.
\begin{lem}\label{lem:21}
The gradient of $W$ is given by  $\nabla W=-(\ol{ f_1},\dots, \ol{f_N})$, where
\[f_j=\sum_{k\neq j} \frac{d_j d_k}{a_j-a_k}.\]

The rational function $G$ given by  
\[
G(z)=\sum_{j\neq k} \frac{d_j d_k}{(z-a_j)(z-a_k)}
\]
has the partial fraction decomposition
\[
G(z)=\sum_{j=1}^N \frac{ f_j}{z-a_j}.
\]
In particular, $G$ is identically zero if and only if the configuration of vortices is in equilibrium.
\end{lem}
\begin{proof}
We compute 
\[\nabla_a \log\frac1{| a-b|}= -\frac12\nabla_a \log|a-b|^2 = -\frac{a-b}{|a-b|^2}
 = -\frac 1{\ol {a - b}},\]
so the claim about the gradient follows by summation. 

The function $G$ clearly only has single poles, and the residue at a potential pole $a_j$ is $\sum_{k\neq j}\frac{d_jd_k}{a_j-a_k}=f_j$.
\end{proof}
\begin{lem}\label{lem:22}
For any $p,q\in \R^2=\C$ and any $\eps<\frac12|p-q|$, we have
\[
\int_{\R^2\setminus \left(B_{\eps}(p)\cup B_\eps(q)\right) } \frac1{(\ol{z-p})^2} \frac1{(z-q)^2} dx dy = 0.
\]
\end{lem}
\begin{proof}
As in \cite{Espo}, the idea is to compute the integral in polar coordinates. In \cite{Espo}, this is done for a domain 
where several $\eps$-disks have been removed from the plane, using 
auxiliary functions that allow small annular regions instead of small disks to be removed without affecting the limit. 
In our domain of integration, only two disks have been removed, so we can instead use change of variables with a M\"obius
transformation
to make the entire domain radially symmetric without introducing auxiliary functions.

By homothety, it suffices to consider $p=0$ and $q=1$. 
We look for a M\"obius transformation $T$ mapping $B_{R_2}(0)\setminus \ol{B_{R_1}(0)}$ to $\R^2 \setminus (B_\eps(0)\cup B_\eps(1))$. We can find its inverse by using the ansatz $S(z)=\frac{z-a}{b-z}$ with $R_1=S(\eps)=-S(-\eps)$, $R_2=S(1-\eps)=-S(1+\eps)$. In particular, $S$
maps $\partial B_\e(0)$ onto $B_{R_1}(0)$, $\partial B_{\e}(1)$ onto $\partial{B_{R_2(0)}}$, and the interval $(\eps, 1-\eps)$ onto the interval $(R_1,R_2)$. 
This leads to $ab=\eps^2$ and $a+b=1$, 
so $a=\frac12(1-\sqrt{1-4\e^2})$ and $b=\frac12(1+\sqrt{1-4\e^2})$. Note that for $\eps\to 0$, $a=\e^2+O(\e^4)$ and $b= 1-\e^2+O(\e^4)$. 
We have $R_1=S(\e)=\frac{\e-a}{b-\e} =O(\e)$ and $R_2=S(1-\e)=\frac1{R_1}=O(\frac1\e)$. 
The inverse of $S$ is given by $T(w)=\frac{bw+a}{1+w}$. Clearly $T'(w)=\frac{b(w+1)-(bw+a)}{(w-1)^2}=\frac{b-a}{(w+1)^2}$.

Changing variables to $w=u+iv$ and using $dxdy=|T'(w)|^2 dudv$, 
\begin{multline*}
I=\int_{\R^2\setminus \left(B_{\e}(0)\cup B_\e(1)\right) } \frac1{\ol{z}^2} \frac1{(z-1)^2} dx dy \\= 
\int_{B_{R_2}(0)\setminus \ol{B_{R_1}(0)}} \frac1{\left(\ol{\frac{wb+a}{w+1}}\right)^2} \frac1{(\frac{wb+a}{w+1}-1)^2} \frac{|b-a|^2}{|w+1|^4} dudv
\\= \int_{B_{R_2}(0)\setminus \ol{B_{R_1}(0)}}
\frac{|b-a|^2}{(\ol w b+a)^2(w(b-1)+a-1)^2}
dudv
\end{multline*}
For $|w|=r$, we write $w=r\zeta=re^{i\theta}$ and $\ol w =\frac r\zeta$. In polar coordinates and using $d\zeta=i\zeta d\theta$ as well 
as $a+b=1$, 
\[
I=\int_{R_1}^{R_2} \int_{\partial B_1} \frac{r \zeta|b-a|^2}{ i (rb+a\zeta)^2 (ra\zeta+b)^2} d\zeta dr
\]
Computing the $\zeta$-integral using the residue theorem, we note that the integrand has singularities for $|\zeta|<1$ 
if $\frac{rb}{a}<1$ or $\frac{b}{ra}<1$. However, from the monotonicity of $S$ on $(a,b)$ we infer $\frac{a}{b}=S(2ab)=S(2\e^2)<S(\e)=R_1<r<R_2<\frac ba$, so there are no singularities and the $\zeta$-integral is zero for every $r$, and we must have $I=0$. 
\end{proof}
\begin{thm}\label{thm:t}
If the configuration $(a_1,d_1,\dots, a_N, d_N)$ is in equilibrium, then
\[\lim_{\eps\to 0} A_\eps(a_1,d_1,\dots, a_N, d_N)=0.\]
\end{thm}
\begin{proof}
We begin by noting that
\[
\left|\sum_{j=1}^N \frac{d_j}{z-a_j}\right|^4 = \left| \left( \sum_{j=1}^N \frac{d_j}{z-a_j}\right)^2
\right|^2 = 
\left|\sum_{j=1}^N \frac{d_j^2}{(z-a_j)^2} + \sum_{j\neq k} \frac{d_jd_k}{(z-a_j)(z-a_k)}
 \right|^2 .
\]
If the vortices are in equilibrium then by Lemma \ref{lem:21}, the cross term is $G(z)\equiv 0$. We can thus write, using the 
abbreviation $T_j=\frac{d_j^2}{(z-a_j)^2}$,
\[
\left|\sum_{j=1}^N \frac{d_j}{z-a_j}\right|^4-\sum_{j=1}^N \frac{d_j^4}{|z-a_j|^4} 
= \left| \sum_{j=1}^N T_j\right|^2 - \sum_{j=1}^N |T_j|^2 = \sum_{j \neq k} \ol{T_j} T_k.
\]
Now for $\ell\notin\{j,k\}$, clearly $\int_{B_\eps(a_\ell)} \ol{T_j} T_k dx dy\to 0$ as $\e\to 0$ so 
\[
\lim_{\eps\to 0}\int_{\R^2\setminus \bigcup_{\ell=1}^N B_\eps(a_\ell)} \ol{T_j} T_k dxdy = 
\lim_{\eps\to 0}\int_{\R^2\setminus (B_\eps(a_j)\cup B_\eps(a_k))} \ol T_j T_k dxdy =0,
\]
where we have used Lemma \ref{lem:22}. It follows that 
\[
A=\lim_{\eps\to 0}\int_{\R^2\setminus \bigcup_{\ell=1}^N B_\eps(a_\ell) } \left(\left|\sum_{j=1}^N \frac{d_j}{z-a_j}\right|^4-\sum_{j=1}^N \frac{d_j^4}{|z-a_j|^4} 
 \right)dxdy =0.
\]
\end{proof}

\end{document}